\documentclass[reqno]{amsart}
\usepackage{amsmath}
\usepackage{hyperref}
\usepackage{amssymb}

\AtBeginDocument{{\noindent\small
\emph{}
}
\
\vspace{9mm}}
\input{tcilatex}

\begin{document}
\title[Positive solutions to boundary-value problems]{Existence and
uniqueness of positive solutions of boundary-value problems for fractional
differential equations with $p$-Laplacian operator}
\author[E. \c{S}en, M. Acikgoz, J. J. Seo, S. Araci, K. Oru\c{c}o\u{g}lu]{%
Erdo\u{g}an \c{S}en, Mehmet Acikgoz, Jong Jin Seo, Serkan Araci, Kamil Oru%
\c{c}o\u{g}lu}
\address{Erdo\u{g}an \c{S}en \\
Department of Mathematics, Faculty of Science and Letters, Namik Kemal
University, 59030 Tekirda\u{g}, Turkey\\
Department of Mathematics Engineering, Istanbul Technical University,
Maslak, 34469 Istanbul, Turkey}
\email{erdogan.math@gmail.com}
\address{Mehmet Acikgoz \\
University of Gaziantep, Faculty of Science and Arts, Department of
Mathematics, 27310 Gaziantep}
\email{acikgoz@gantep.edu.tr}
\address{Jong Jin Seo\\
Department of Applied Mathematics, Pukyong National University, Busan
608-737, Republic of Korea}
\email{seo2011@pknu.ac.kr}
\address{Serkan Araci \\
University of Gaziantep, Faculty of Science and Arts, Department of
Mathematics, 27310 Gaziantep}
\email{mtsrkn@hotmail.com}
\address{Kamil Oru\c{c}o\u{g}lu\\
Department of Mathematics Engineering, Istanbul Technical University,
Maslak, 34469 Istanbul, Turkey}
\email{koruc@itu.edu.tr}
\subjclass[2000]{34A08, 34B18, 35J05}
\keywords{Fractional boundary-value problem; positive solution; cone;
\hfill\break\indent 
Schauder fixed point theorem; uniqueness; $p$-Laplacian operator}

\begin{abstract}
In this article, we consider the following boundary-value problem of
nonlinear fractional differential equation with $p$-Laplacian operator 
\begin{gather*}
D_{0+}^{\beta }(\phi _{p}(D_{0+}^{\alpha }u(t)))+a(t)f(u)=0,\quad 0<t<1, \\
u(0)=\gamma u(h)+\lambda ,\text{ }u^{\prime }(0)=\mu ,\quad \\
\phi _{p}(D_{0+}^{\alpha }u(0))=(\phi _{p}(D_{0+}^{\alpha }u(1)))^{\prime
}=(\phi _{p}(D_{0+}^{\alpha }u(0)))^{\prime \prime }=(\phi
_{p}(D_{0+}^{\alpha }u(0)))^{\prime \prime \prime }=0,
\end{gather*}%
where $1<\alpha \leqslant 2$, $3<\beta \leqslant 4$ are real numbers, $%
D_{0+}^{\alpha },D_{0+}^{\beta }$ are the standard Caputo fractional
derivatives, $\phi _{p}(s)=|s|^{p-2}s$, $p>1$, $\phi _{p}^{-1}=\phi _{q}$, $%
1/p+1/q=1$, $0\leqslant \gamma <1$, $0\leqslant h\leqslant 1$, $\lambda ,\mu
>0$ are parameters, $a:(0,1)\rightarrow \lbrack 0,+\infty )$ and $%
f:[0,+\infty )\rightarrow \lbrack 0,+\infty )$ are continuous. By the
properties of Green function and Schauder fixed point theorem, several
existence and nonexistence results for positive solutions, in terms of the
parameters $\lambda $ and $\mu $ are obtained. The uniqueness of positive
solution on the parameters $\lambda $ and $\mu $ is also studied. In the
final section of this paper, we derive not only new but also interesting
identities related special polynomials by which Caputo fractional derivative.
\end{abstract}

\maketitle


\allowdisplaybreaks

\section{Introduction}

In 1695, L'H\^{o}pital asked Leibniz: What if the order of the derivative is 
$\frac{1}{2}$? To which Leibniz considered in an useful means, "thus it
follows that will be equal to $x\sqrt{dx:x}$, an obvious paradox. In recent
years, fractional calculus has been studied by many mathematicians from
Leibniz's time to the present.

Also, fractional differential equations arise in many engineering and
scientific disciplines as the mathematical modelling of systems and
processes in the fields of physics, fluid flows, electrical networks,
viscoelasticity, aerodynamics, and many other branches of science. For
details, see \cite{Agarwal,balan2,Balan1,Kilbas, machado, Meral,
Oldham,Podlubny,Weitzner}.

In the last few decades, fractional-order models are found to be more
adequate than integer-order models for some real world problems. Recently,
there have been some papers dealing with the existence and multiplicity of
solutions (or positive solutions) of non linear initial fractional
differential equations by the use of techniques of nonlinear analysis \cite%
{bai, chai,Chen21,Feng17, wang12,Wang14, xu16,Xu20, yang19, Zhao15, Zhao18,
zhou13}, upper and lower solutions method \cite{Liang,Lu,Wang}, fixed point
index \cite{Dix,Xu}, coincidence theory \cite{Chen22}, Banach contraction
mapping principle \cite{Liu}, etc).

Chai \cite{chai} investigated the existence and multiplicity of positive
solutions for a class of boundary-value problem of fractional differential
equation with $p$-Laplacian operator 
\begin{gather*}
D_{0+}^{\beta }(\phi _{p}(D_{0+}^{\alpha }u(t)))+f(t,u(t),D_{0+}^{\rho
}u(t))=0,\quad 0<t<1, \\
u(0)=0,u(1)+\sigma D_{0+}^{\gamma }u(1)=0,\quad D_{0+}^{\alpha }u(0)=0,
\end{gather*}%
where $1<\alpha \leq 2,0<\gamma \leq 1$, $0\leq \alpha -\gamma -1$, $\sigma $
is a positive constant number, $D_{0+}^{\alpha },D_{0+}^{\beta
},D_{0+}^{\gamma }$ are the standard Riemann-Liouville derivatives. By means
of the fixed point theorem on cones, some existence and multiplicity results
of positive solutions are obtained.

Although the fractional differential equation boundary-value problems have
been studied by several authors, very little is known in the literature on
the existence and nonexistence of positive solutions of fractional
differential equation boundary-value problems with $p$-Laplacian operator
when a parameter $\lambda $ is involved in the boundary conditions. We also
mention that, there is very little known about the uniqueness of the
solution of fractional differential equation boundary-value problems with $p$%
-Laplacian operator on the parameter $\lambda $. Han et al [33] studied the
existence and uniqueness of positive solutions for the fractional
differential equation with $p$-Laplacian operator%
\begin{gather*}
D_{0+}^{\beta }(\phi _{p}(D_{0+}^{\alpha }u(t)))+a(t)f(u)=0,\quad 0<t<1, \\
u(0)=\gamma u(\xi )+\lambda ,\quad \phi _{p}(D_{0+}^{\alpha }u(0))=(\phi
_{p}(D_{0+}^{\alpha }u(1)))^{\prime }=(\phi _{p}(D_{0+}^{\alpha
}u(0)))^{\prime \prime }=0.
\end{gather*}%
where $0<\alpha \leqslant 1$, $2<\beta \leqslant 3$ are real numbers; $%
D_{0+}^{\alpha },D_{0+}^{\beta }$ are the standard Caputo fractional
derivatives, $\phi _{p}(s)=|s|^{p-2}s$, $p>1$. Therefore, to enrich the
theoretical knowledge of the above, in this paper, we investigate the
following $p$-Laplacian fractional differential equation boundary-value
problem 
\begin{gather}
D_{0+}^{\beta }(\phi _{p}(D_{0+}^{\alpha }u(t)))+a(t)f(u)=0,\quad 0<t<1, 
\tag{1} \\
u(0)=\gamma u(h)+\lambda ,\text{ }u^{\prime }(0)=\mu ,\quad  \notag \\
\phi _{p}(D_{0+}^{\alpha }u(0))=(\phi _{p}(D_{0+}^{\alpha }u(1)))^{\prime
}=(\phi _{p}(D_{0+}^{\alpha }u(0)))^{\prime \prime }=(\phi
_{p}(D_{0+}^{\alpha }u(0)))^{\prime \prime \prime }=0,  \tag{2}
\end{gather}%
where $1<\alpha \leqslant 2$, $3<\beta \leqslant 4$ are real numbers, $%
D_{0+}^{\alpha },D_{0+}^{\beta }$ are the standard Caputo fractional
derivatives, $\phi _{p}(s)=|s|^{p-2}s$, $p>1$, $\phi _{p}^{-1}=\phi _{q}$, $%
1/p+1/q=1$, $0\leqslant \gamma <1$, $0\leqslant h\leqslant 1$, $\lambda ,\mu
>0$ are parameters, $a:(0,1)\rightarrow \lbrack 0,+\infty )$ and $%
f:[0,+\infty )\rightarrow \lbrack 0,+\infty )$ are continuous. By the
properties of Green function and Schauder fixed point theorem, several
existence and nonexistence results for positive solutions, in terms of the
parameters $\lambda $ and $\mu $ are obtained. The uniqueness of positive
solution on the parameters $\lambda $ and $\mu $ is also studied.

\section{Preliminaries and related lemmas}

\begin{definition}[\cite{Kilbas}]\label{d2.1}\rm
The Riemann-Liouville fractional integral of order $\alpha>0$ of a
function $y:(0,+\infty)\to \mathbb{R}$ is given by
$$
I_{0+}^\alpha y(t)=\frac{1}{\Gamma(\alpha)}\int_{0}^t(t-s)^{\alpha-1}y(s)ds
$$
provided the right side is pointwise defined on $(0,+\infty)$.
\end{definition}

\begin{definition}[\cite{Kilbas}]\label{d2.2}\rm
 The Caputo fractional derivative of order
$\alpha>0$ of a continuous function $y:(0,+\infty)\to \mathbb{R}$ is
given by
$$
D_{0+}^\alpha y(t)=\frac{1}{\Gamma(n-\alpha)}
\int_0^t\frac{y^{(n)}(s)}{(t-s)^{\alpha-n+1}}ds,
$$
where $n$ is the smallest integer greater than or equal to $\alpha$,
provided that the right side is pointwise defined on $(0,+\infty)$. 
\end{definition}

\begin{remark}[\cite{Podlubny}]\label{r2.1}\rm
By Definition \ref{d2.2}, under natural
conditions on the function $f(t)$, for $\alpha\to n$ the Caputo
derivative becomes a conventional $n$-th derivative of the function
$f(t)$.
\end{remark}

\begin{remark}[\cite{Kilbas}]\label{r2.2} \rm
As a basic example, 
$$
D_{0^+}^\alpha t^\mu=\mu(\mu-1)\dots
(\mu-n+1)\frac{\Gamma(1+\mu-n)}{\Gamma(1+\mu-\alpha)}t^{\mu-\alpha},\quad
\text {for } t\in(0,\infty).
$$ 
In particular $D_{0^+}^\alpha t^\mu=0$,
$\mu=0,1,\dots,n-1$, where $D^\alpha_{0^+}$ is the Caputo
fractional derivative, $n$ is the smallest integer greater than or
equal to $\alpha$.
\end{remark}

From the definition of the Caputo derivative and Remark \ref{r2.2}, we can
obtain the following statement.

\begin{lemma}[\cite{Kilbas}]\label{l2.1}\rm
 Let $\alpha>0$. Then the fractional differential equation
$$
D_{0+}^\alpha u(t)=0
$$
has a unique solution
$$
u(t)=c_0+c_1t+c_2t^2+\dots+c_{n-1}t^{n-1}, \quad c_i\in \mathbb{R},\,
i=0,1,2,\dots,n-1, 
$$
where $n$ is the smallest integer greater than or equal to $\alpha$.
\end{lemma}

\begin{lemma}[\cite{Kilbas}]\label{l2.2}\rm
Let $\alpha>0$. Assume that $u\in C^n[0,1]$. Then
$$
I_{0+}^\alpha D_{0+}^\alpha u(t)=u(t)+c_0+c_1t+c_2t^2+\dots+c_{n-1}t^{n-1},
$$
for some $c_i\in \mathbb{R}$, $i=0,1,2,\dots,n-1$, where $n$ is the
smallest integer greater than or equal to $\alpha$.
\end{lemma}

\begin{lemma}
Let $y\in C\left[ 0,1\right] $ and$1<\alpha \leq 2.$ Then fractional
differential equation boundary-value problem 
\begin{equation}
D_{0+}^{\alpha }u\left( t\right) =y(t),\text{ }0<t<1  \tag{3}
\end{equation}%
\begin{equation}
u(0)=\gamma u(h)+\lambda ,\text{ }u^{\prime }(0)=\mu   \tag{4}
\end{equation}%
has a unique solution 
\begin{equation*}
u(t)=\int_{0}^{t}\frac{(t-s)^{\alpha -1}}{\Gamma (\alpha )}y(s)ds+\frac{%
\gamma }{1-\gamma }\int_{0}^{h}\frac{(h-s)^{\alpha -1}}{\Gamma (\alpha )}%
y(s)ds+\frac{\lambda +\gamma \mu h}{1-\gamma }.
\end{equation*}
\end{lemma}

\begin{proof}
We apply Lemma \ref{l2.2} to reduce (3) to an equivalent integral equation, 
\begin{equation*}
u(t)=I_{0+}^{\alpha }y(t)+c_{0}+c_{1}t,\quad c_{0},c_{1}\in \mathbb{R}.
\end{equation*}%
Consequently, the general solution of (3) is 
\begin{equation*}
u(t)=\int_{0}^{t}\frac{(t-s)^{\alpha -2}}{\Gamma (\alpha )}%
y(s)ds+c_{0}+c_{1}t,\quad c_{0},c_{1}\in \mathbb{R}.
\end{equation*}%
By (4), we has 
\begin{equation*}
c_{0}=\frac{\gamma }{1-\gamma }\int_{0}^{h}\frac{(h-s)^{\alpha -2}}{\Gamma
(\alpha )}y(s)ds+\frac{\gamma c_{1}h}{1-\gamma }+\frac{\lambda }{1-\gamma },
\end{equation*}%
and since $u^{\prime }(t)=c_{1}$, we have by (4)%
\begin{equation*}
c_{1}=\mu
\end{equation*}%
Therefore, the unique solution of problem (3) and (4) is 
\begin{equation*}
u(t)=\int_{0}^{t}\frac{(t-s)^{\alpha -2}}{\Gamma (\alpha )}y(s)ds+\frac{%
\gamma }{1-\gamma }\int_{0}^{h}\frac{(h-s)^{\alpha -2}}{\Gamma (\alpha )}%
y(s)ds+\frac{\lambda }{1-\gamma }.+\frac{\gamma \mu h}{1-\gamma }.
\end{equation*}
\end{proof}

\begin{lemma}
Let $y\in C\left[ 0,1\right] $ and$1<\alpha \leq 2,$ $3<\beta \leq 4.$ Then
fractional differential equation boundary-value problem 
\begin{equation}
D_{0+}^{\beta }(\phi _{p}(D_{0+}^{\alpha }u(t)))+y(t)=0,\quad 0<t<1,  \tag{5}
\end{equation}%
\begin{equation}
\left\{ 
\begin{array}{c}
u(0)=\gamma u(h)+\lambda ,\text{ }u^{\prime }(0)=\mu , \\ 
\phi _{p}(D_{0+}^{\alpha }u(0))=(\phi _{p}(D_{0+}^{\alpha }u(1)))^{\prime
}=(\phi _{p}(D_{0+}^{\alpha }u(0)))^{\prime \prime }=(\phi
_{p}(D_{0+}^{\alpha }u(0)))^{\prime \prime \prime }=0,%
\end{array}%
\right.   \tag{6}
\end{equation}%
has a unique solution 
\begin{eqnarray*}
u(t) &=&\int_{0}^{t}\frac{(t-s)^{\alpha -1}}{\Gamma (\alpha )}\phi
_{q}\left( \int_{0}^{1}H(s,\tau )y(\tau )d\tau \right) ds \\
&&+\frac{\gamma }{1-\gamma }\int_{0}^{h}\frac{(h-s)^{\alpha -1}}{\Gamma
(\alpha )}\phi _{q}\left( \int_{0}^{1}H(s,\tau )y(\tau )d\tau \right) ds+%
\frac{\lambda +\gamma \mu h}{1-\gamma },
\end{eqnarray*}%
where%
\begin{equation*}
H\left( t,s\right) =\left\{ 
\begin{array}{cc}
\frac{t\left( \beta -1\right) \left( 1-s\right) ^{\beta -2}-\left(
t-s\right) ^{\beta -1}}{\Gamma (\beta )}, & 0\leq s\leq t\leq 1, \\ 
\frac{t\left( \beta -1\right) \left( 1-s\right) ^{\beta -2}}{\Gamma (\beta )}%
, & 0\leq t\leq s\leq 1.%
\end{array}%
\right. 
\end{equation*}
\end{lemma}

\begin{proof}
From Lemma \ref{l2.2}, the boundary-value problem (5) and (6) is equivalent
to the integral equation 
\begin{equation*}
\phi _{p}(D_{0+}^{\alpha }u(t))=-I_{0+}^{\beta
}y(t)+c_{0}+c_{1}t+c_{2}t^{2}+c_{3}t^{3},
\end{equation*}%
for some $c_{0},c_{1},c_{2},c_{3}\in \mathbb{R}$; that is, 
\begin{equation*}
\phi _{p}(D_{0+}^{\alpha }u(t))=-\int_{0}^{t}\frac{(t-\tau )^{\beta -1}}{%
\Gamma (\beta )}y(\tau )d\tau +c_{0}+c_{1}t+c_{2}t^{2}+c_{3}t^{3}.
\end{equation*}%
By the boundary conditions $\phi _{p}(D_{0+}^{\alpha }u(0))=(\phi
_{p}(D_{0+}^{\alpha }u(1)))^{\prime }=(\phi _{p}(D_{0+}^{\alpha
}u(0)))^{\prime \prime }=(\phi _{p}(D_{0+}^{\alpha }u(0)))^{\prime \prime
\prime }=0$, we have 
\begin{equation*}
c_{0}=c_{2}=c_{3}=0,\quad c_{1}=\int_{0}^{1}\frac{(\beta -1)(1-\tau )^{\beta
-2}}{\Gamma (\beta )}y(\tau )d\tau .
\end{equation*}%
Therefore, the solution $u(t)$ of fractional differential equation
boundary-value problem (5) and (6) satisfies 
\begin{align*}
\phi _{p}(D_{0+}^{\alpha }u(t))& =-\int_{0}^{t}\frac{(t-\tau )^{\beta -1}}{%
\Gamma (\beta )}y(\tau )d\tau +t\int_{0}^{1}\frac{(\beta -1)(1-\tau )^{\beta
-2}}{\Gamma (\beta )}y(\tau )d\tau \\
& =\int_{0}^{1}H(t,\tau )y(\tau )d\tau .
\end{align*}%
Consequently, $D_{0+}^{\alpha }u(t)=\phi _{q}\Big(\int_{0}^{1}H(t,\tau
)y(\tau )d\tau \Big)$. Thus, fractional differential equation boundary-value
problem (5) and (6) is equivalent to the problem 
\begin{gather*}
D_{0+}^{\alpha }u(t)=\phi _{q}\Big(\int_{0}^{1}H(t,\tau )y(\tau )d\tau \Big)%
,\quad 0<t<1, \\
u(0)=\gamma u(h)+\lambda ,\text{ }u^{\prime }(0)=\mu .
\end{gather*}%
Lemma 3 implies that fractional differential equation boundary-value problem
(5) and (6) has a unique solution, 
\begin{align*}
u(t)& =\int_{0}^{t}\frac{(t-s)^{\alpha -2}}{\Gamma (\alpha )}\phi _{q}\Big(%
\int_{0}^{1}H(s,\tau )y(\tau )d\tau \Big)ds \\
& \quad +\frac{\gamma }{1-\gamma }\int_{0}^{h}\frac{(h-s)^{\alpha -2}}{%
\Gamma (\alpha )}\phi _{q}\Big(\int_{0}^{1}H(s,\tau )y(\tau )d\tau \Big)ds+%
\frac{\lambda +\gamma \mu h}{1-\gamma }.
\end{align*}%
The proof is complete.
\end{proof}

\begin{lemma}[\cite{Wang14}]\label{l2.5}
Let $1<\alpha\leqslant2,3<\beta\leqslant4$. The function $H(t,s)$ 
is
continuous on $[0,1]\times[0,1]$ and satisfies
\begin{itemize}
\item[(1)] $H(t,s)\geqslant0, H(t,s)\leqslant H(1,s)$, \; for $t,s\in[0,1]$;

\item[(2)] $H(t,s)\geqslant t^{\beta-1}H(1,s)$, \; for $t,s\in(0,1)$.
\end{itemize}
\end{lemma}

\begin{lemma}[Schauder fixed point theorem \cite{diethelm}] \label{l2.6}
Let $(E,d)$ be a complete metric space, $U$ be
a closed convex subset of $E$, and $A : U\to U$ be a mapping such that the 
set $\{Au : u\in U\}$ is relatively compact in $E$. Then $A$ has at 
least one fixed point.
\end{lemma}

To prove our main results, we use the following assumptions.

\begin{itemize}
\item[(H1)] $0<\int_0^1H(1,\tau)a(\tau)d\tau<+\infty$;

\item[(H2)] there exist $0<\sigma <1$ and $c>0$ such that 
\begin{equation}
f(x)\leqslant \sigma L\phi _{p}(x),\quad \text{for }0\leqslant x\leqslant c,
\tag{7}  \label{e2.7}
\end{equation}%
where $L$ satisfies 
\begin{equation}
0<L\leqslant \Big[\phi _{p}\Big(\frac{1+\gamma (h^{\alpha }-1)}{\Gamma
(\alpha +1)(1-\gamma )}\Big)\int_{0}^{1}H(1,\tau )a(\tau )d\tau \Big]^{-1}; 
\tag{8}  \label{e2.8}
\end{equation}

\item[(H3)] there exist $d>0$ such that 
\begin{equation}
f(x)\leqslant M\phi _{p}(x),\quad \text{for }d<x<+\infty ,  \tag{9}
\label{e2.9}
\end{equation}%
where $M$ satisfies 
\begin{equation}
0<M<\Big[\phi _{p}\Big(\frac{1+\gamma (h^{\alpha }-1)}{\Gamma (\alpha
+1)(1-\gamma )}2^{q-1}\Big)\int_{0}^{1}H(1,\tau )a(\tau )d\tau \Big]^{-1}; 
\tag{10}  \label{e2.10}
\end{equation}

\item[(H4)] there exist $0<\delta <1$ and $e>0$ such that 
\begin{equation}
f(x)\geqslant N\phi _{p}(x),\quad \text{for }e<x<+\infty ,  \tag{11}
\label{e2.11}
\end{equation}%
where $N$ satisfies 
\begin{equation}
N>\Big[\phi _{p}\Big(c_{\delta }\int_{0}^{1}\frac{(1-s)^{\alpha -2}}{\Gamma
(\alpha )}\phi _{q}(s^{\beta -1})ds\Big)\int_{\delta }^{1}H(1,\tau )a(\tau
)d\tau \Big]^{-1};  \tag{12}  \label{e2.12}
\end{equation}%
with 
\begin{equation}
c_{\delta }=\int_{0}^{\delta }\alpha (1-s)^{\alpha -2}\phi _{q}(s^{\beta
-1})ds\in (0,1);  \tag{13}  \label{e2.13}
\end{equation}

\item[(H5)] $f(x)$ is nondecreasing in $x$;

\item[(H6)] there exist $0\leqslant \theta <1$ such that 
\begin{equation}
f(kx)\geqslant (\phi _{p}(k))^{\theta }f(x),\quad \text{for any $0<k<1$ and $%
0<x<+\infty $}.  \tag{14}  \label{e2.14}
\end{equation}
\end{itemize}

\begin{remark}\label{r2.1b} \rm
Let
$$
f_0=\lim_{x\to0^+}\frac{f(x)}{\phi_p(x)}, \quad 
f_\infty=\lim_{x\to+\infty}\frac{f(x)}{\phi_p(x)}.
$$
Then, (H2) holds if $f_0=0$, (H3) holds if $f_\infty=0$,
 and (H4) holds if $f_\infty=+\infty$.
\end{remark}

\section{Existence}

\begin{theorem}
Assume that (H1), (H2) hold. Then the fractional differential equation
boundary-value problem (1.1) and (1.2) has at least one positive solution
for $0<\lambda +\gamma \mu \leq \left( 1-\gamma \right) \left( 1-\phi
_{q}\left( \sigma \right) \right) c.$
\end{theorem}

\begin{proof}
Let $c>0$ be given in (H2). Define 
\begin{equation*}
K_{1}=\{u\in C[0,1]:0\leqslant u(t)\leqslant c\text{ on }[0,1]\}
\end{equation*}%
and an operator $T_{\lambda }:K_{1}\rightarrow C[0,1]$ by%
\begin{equation*}
T_{\lambda }u\left( t\right) =\int_{0}^{t}\frac{\left( t-s\right) ^{\alpha
-2}}{\Gamma \left( \alpha \right) }\phi _{q}\left( \int_{0}^{1}H(s,\tau
)a(\tau )f\left( u\left( \tau \right) \right) d\tau \right)
\end{equation*}%
\begin{equation}
+\frac{\gamma }{1-\gamma }\int_{0}^{h}\frac{(h-s)^{\alpha -2}}{\Gamma
(\alpha )}\phi _{q}\left( \int_{0}^{1}H(s,\tau )a(\tau )f\left( u\left( \tau
\right) \right) d\tau \right) ds+\frac{\lambda +\gamma \mu h}{1-\gamma } 
\tag{15}  \label{e3.1}
\end{equation}%
Then, $K_{1}$ is a closed convex set. From Lemma 4, $u$ is a solution of
fractional differential equation boundary-value problem (1) and (2) if and
only if $u$ is a fixed point of $T_{\lambda }$. Moreover, a standard
argument can be used to show that $T_{\lambda }$ is compact.

For any $u\in K_{1}$, from (3) and (4), we obtain 
\begin{equation*}
f(u(t))\leqslant \sigma L\phi _{p}(u(t))\leqslant \sigma L\phi _{p}(c),\quad 
\text{on }[0,1],
\end{equation*}%
and 
\begin{equation*}
\frac{1+\gamma (h^{\alpha }-1)}{\Gamma (\alpha +1)(1-\gamma )}\phi
_{q}(L)\phi _{q}\Big(\int_{0}^{1}H(1,\tau )a(\tau )d\tau \Big)\leqslant 1.
\end{equation*}%
Let $0<\lambda +\gamma \mu \leqslant (1-\gamma )(1-\phi _{q}(\sigma ))c$.
Then, from Lemma \ref{l2.5} and \eqref{e3.1}, it follows that 
\begin{align*}
0\leqslant T_{\lambda }u(t)& \leqslant \int_{0}^{t}\frac{(t-s)^{\alpha -2}}{%
\Gamma (\alpha )}\phi _{q}\Big(\int_{0}^{1}H(1,\tau )a(\tau )f(u(\tau
))d\tau \Big)ds \\
& \quad +\frac{\gamma }{1-\gamma }\int_{0}^{h}\frac{(h-s)^{\alpha -2}}{%
\Gamma (\alpha )}\phi _{q}\Big(\int_{0}^{1}H(1,\tau )a(\tau )f(u(\tau
))d\tau \Big)ds+\frac{\lambda +\gamma \mu h}{1-\gamma } \\
& \leqslant \frac{1}{\Gamma (\alpha +1)}\phi _{q}\Big(\int_{0}^{1}H(1,\tau
)a(\tau )f(u(\tau ))d\tau \Big) \\
& \quad +\frac{\gamma h^{\alpha }}{\Gamma (\alpha +1)(1-\gamma )}\phi _{q}%
\Big(\int_{0}^{1}H(1,\tau )a(\tau )f(u(\tau ))d\tau \Big)+(1-\phi
_{q}(\sigma ))c \\
& =\frac{1+\gamma h^{\alpha }-\gamma )}{\Gamma (\alpha +1)(1-\gamma )}\phi
_{q}\Big(\int_{0}^{1}H(1,\tau )a(\tau )f(u(\tau ))d\tau \Big)+(1-\phi
_{q}(\sigma ))c \\
& \leqslant \frac{1+\gamma (h^{\alpha }-1)}{\Gamma (\alpha +1)(1-\gamma )}%
\phi _{q}(L)\phi _{q}\Big(\int_{0}^{1}H(1,\tau )a(\tau )d\tau \Big)\phi
_{q}(\sigma )c+(1-\phi _{q}(\sigma ))c \\
& \leqslant \phi _{q}(\sigma )c+(1-\phi _{q}(\sigma ))c=c,\quad t\in \lbrack
0,1].
\end{align*}%
Thus, $T_{\lambda }(K_{1})\subseteq K_{1}$, By Schauder fixed point theorem, 
$T_{\lambda }$ has a fixed point $u\in K_{1}$; that is, the fractional
differential equation boundary-value problem (1) and (2) has at least one
positive solution. The proof is complete.
\end{proof}

\begin{corollary}
Assume that (H1) holds and $f_{0}=0$. Then the fractional differential
equation boundary-value problem (1) and (2) has at least one positive
solution for sufficiently small $\lambda >0.$ 
\end{corollary}

\begin{theorem}
Assume that (H1), (H3) hold. Then the fractional differential equation
boundary-value problem (1) and (2) has at least one positive solution for
all $\lambda >0.$ 
\end{theorem}

\begin{proof}
Let $\lambda >0$ be fixed and $d>0$ be given in (H3). Define $%
D=\max_{0\leqslant x\leqslant d}f(x)$. Then 
\begin{equation}
f(x)\leqslant D,\quad \text{for }0\leqslant x\leqslant d.  \tag{16}
\label{e3.2}
\end{equation}%
From \eqref{e2.10}, we have 
\begin{equation*}
\frac{1+\gamma (h^{\alpha }-1)}{\Gamma (\alpha +1)(1-\gamma )}2^{q-1}\phi
_{q}(M)\phi _{q}\Big(\int_{0}^{1}H(1,\tau )a(\tau )d\tau \Big)<1.
\end{equation*}%
Thus, there exists $d^{\ast }>d$ large enough so that 
\begin{equation}
\frac{1+\gamma (h^{\alpha }-1)}{\Gamma (\alpha +1)(1-\gamma )}2^{q-1}(\phi
_{q}(D)+\phi _{q}(M)d^{\ast })\phi _{q}\Big(\int_{0}^{1}H(1,\tau )a(\tau
)d\tau \Big)+\frac{\lambda +\gamma \mu h}{1-\gamma }\leqslant d^{\ast }. 
\tag{17}  \label{e3.3}
\end{equation}%
Let 
\begin{equation*}
K_{2}=\{u\in C[0,1]:0\leqslant u(t)\leqslant d^{\ast }\text{ on }[0,1]\}.
\end{equation*}%
For $u\in K_{2}$, define 
\begin{gather*}
I_{1}^{u}=\{t\in \lbrack 0,1]:0\leqslant u(t)\leqslant d\}, \\
I_{2}^{u}=\{t\in \lbrack 0,1]:d<u(t)\leqslant d^{\ast }\}.
\end{gather*}%
Then, $I_{1}^{u}\cup I_{2}^{u}=[0,1],I_{1}^{u}\cap I_{2}^{u}=\emptyset $,
and in view of \eqref{e2.9}, we have 
\begin{equation}
f(u(t))\leqslant M\phi _{p}(u(t))\leqslant M\phi _{p}(d^{\ast }),\quad \text{%
for }t\in I_{2}^{u}.  \tag{18}  \label{e3.4}
\end{equation}%
Let the compact operator $T_{\lambda }$ be defined by \eqref{e3.1}. Then
from Lemma \ref{l2.5}, \eqref{e2.9} and \eqref{e3.2}, we have 
\begin{align*}
0& \leqslant T_{\lambda }u(t) \\
& \leqslant \int_{0}^{t}\frac{(t-s)^{\alpha -2}}{\Gamma (\alpha )}\phi _{q}%
\Big(\int_{0}^{1}H(1,\tau )a(\tau )f(u(\tau ))d\tau \Big)ds \\
& \quad +\frac{\gamma }{1-\gamma }\int_{0}^{h}\frac{(h-s)^{\alpha -2}}{%
\Gamma (\alpha )}\phi _{q}\Big(\int_{0}^{1}H(1,\tau )a(\tau )f(u(\tau
))d\tau \Big)ds+\frac{\lambda +\gamma \mu h}{1-\gamma } \\
& \leqslant \frac{1}{\Gamma (\alpha +1)}\phi _{q}\Big(\int_{0}^{1}H(1,\tau
)a(\tau )f(u(\tau ))d\tau \Big) \\
& \quad +\frac{\gamma h^{\alpha }}{\Gamma (\alpha +1)(1-\gamma )}\phi _{q}%
\Big(\int_{0}^{1}H(1,\tau )a(\tau )f(u(\tau ))d\tau \Big)+\frac{\lambda
+\gamma \mu h}{1-\gamma } \\
& =\frac{1+\gamma (h^{\alpha }-1)}{\Gamma (\alpha +1)(1-\gamma )}\phi _{q}%
\Big(\int_{I_{1}^{u}}H(1,\tau )a(\tau )f(u(\tau ))d\tau
+\int_{I_{2}^{u}}H(1,\tau )a(\tau )f(u(\tau ))d\tau \Big) \\
& \quad +\frac{\lambda +\gamma \mu h}{1-\gamma } \\
& \leqslant \frac{1+\gamma (h^{\alpha }-1)}{\Gamma (\alpha +1)(1-\gamma )}%
\phi _{q}\Big(D\int_{I_{1}^{u}}H(1,\tau )a(\tau )d\tau +M\phi _{p}(d^{\ast
})\int_{I_{2}^{u}}H(1,\tau )a(\tau )d\tau \Big) \\
& \quad +\frac{\lambda +\gamma \mu h}{1-\gamma } \\
& \leqslant \frac{1+\gamma (h^{\alpha }-1)}{\Gamma (\alpha +1)(1-\gamma )}%
\phi _{q}(D+M\phi _{p}(d^{\ast }))\phi _{q}\Big(\int_{0}^{1}H(1,\tau )a(\tau
)d\tau \Big)+\frac{\lambda +\gamma \mu h}{1-\gamma }.
\end{align*}%
From \eqref{e3.3} and the inequality $(a+b)^{r}\leqslant 2^{r}(a^{r}+b^{r})$
for any $a,b,r>0$ (see, for example, \cite{hardy}), we obtain 
\begin{align*}
0& \leqslant T_{\lambda }u(t) \\
& \leqslant \frac{1+\gamma (h^{\alpha }-1)}{\Gamma (\alpha +1)(1-\gamma )}%
2^{q-1}(\phi _{q}(D)+\phi _{q}(M)d^{\ast })\phi _{q}\Big(\int_{0}^{1}H(1,%
\tau )a(\tau )d\tau \Big)+\frac{\lambda +\gamma \mu h}{1-\gamma }\leqslant
d^{\ast }.
\end{align*}%
Thus, $T_{\lambda }:K_{2}\rightarrow K_{2}$. Consequently, by Schauder fixed
point theorem, $T_{\lambda }$ has a fixed point $u\in K_{2}$, that is, the
fractional differential equation boundary-value problem (1) and (2) has at
least one positive solution. The proof is complete.
\end{proof}

\begin{corollary}
Assume that (H1) holds and $f_{\infty }=0$. Then the fractional differential
equation boundary-value problem (1) and (2) has at least one positive
solution for all $\lambda >0.$ 
\end{corollary}

\section{Uniqueness}

\begin{definition}[\cite{guo27}] \label{d4.1}\rm
 A cone $P$ in a real Banach space $X$ is called solid if its interior
 $P^o$ is not empty.
\end{definition}

\begin{definition}[\cite{guo27}] \label{d4.2} \rm
Let $P$ be a solid cone in a real Banach space $X, T : P^o\to P^o$ be
an operator, and $0\leqslant\theta< 1$. Then T is called a
 $\theta$-concave operator if
$$
T(ku)\geqslant k^\theta Tu\quad \text{for any $0<k<1$ and $u\in P^o$}.
$$
\end{definition}

\begin{lemma}[{\cite[Theorem 2.2.6]{guo27}}] \label{l4.1}
  Assume that $P$ is a normal solid cone in a real Banach space $X$,
$0\leqslant\theta< 1$, and $T : P^o\to P^o$ is a $\theta$-concave 
increasing operator. Then $T$ has only one fixed point in $P^o$.
\end{lemma}

\begin{theorem}\label{t4.1} Assume that {\rm (H1), (H5), (H6)} hold. Then
the fractional differential equation boundary-value problem (1)
and (2) has a unique positive solution for any $\lambda>0$.
\end{theorem}

\begin{proof}
Define $P=\{u\in C[0,1]:u(t)\geqslant 0\text{on }[0,1]\}$. Then $P$ is a
normal solid cone in $C[0,1]$ with 
\begin{equation*}
P^{o}=\{u\in C[0,1]:u(t)>0\ \text{ on }[0,1]\}.
\end{equation*}%
For any fixed $\lambda >0$, let $T_{\lambda }:P\rightarrow C[0,1]$ be
defined by \eqref{e3.1}. Define $T:P\rightarrow C[0,1]$ by 
\begin{align*}
Tu(t)& =\int_{0}^{t}\frac{(t-s)^{\alpha -2}}{\Gamma (\alpha )}\phi _{q}\Big(%
\int_{0}^{1}H(s,\tau )a(\tau )f(u(\tau ))d\tau \Big)ds \\
& \quad +\frac{\gamma }{1-\gamma }\int_{0}^{h}\frac{(h-s)^{\alpha -2}}{%
\Gamma (\alpha )}\phi _{q}\Big(\int_{0}^{1}H(s,\tau )a(\tau )f(u(\tau
))d\tau \Big)ds
\end{align*}%
Then from (H5), we have $T$ is increasing in $u\in P^{o}$ and 
\begin{equation*}
T_{\lambda }u(t)=Tu(t)+\frac{\lambda +\gamma \mu h}{1-\gamma }.
\end{equation*}%
Clearly, $T_{\lambda }:P^{o}\rightarrow P^{o}$. Next, we prove that $%
T_{\lambda }$ is a $\theta $-concave increasing operator. In fact, for $%
u_{1},u_{2}\in P$ with $u_{1}(t)\geqslant u_{2}(t)$ on $[0,1]$, we obtain 
\begin{equation*}
T_{\lambda }u_{1}(t)\geqslant Tu_{2}(t)+\frac{\lambda +\gamma \mu h}{%
1-\gamma }=T_{\lambda }u_{2}(t);
\end{equation*}%
i.e., $T_{\lambda }$ is increasing. Moreover, (H6) implies 
\begin{align*}
T_{\lambda }(ku)(t)& \geqslant k^{\theta }\int_{0}^{t}\frac{(t-s)^{\alpha -2}%
}{\Gamma (\alpha )}\phi _{q}\Big(\int_{0}^{1}H(s,\tau )a(\tau )f(u(\tau
))d\tau \Big)ds \\
& \quad +k^{\theta }\frac{\gamma }{1-\gamma }\int_{0}^{h}\frac{(h-s)^{\alpha
-2}}{\Gamma (\alpha )}\phi _{q}\Big(\int_{0}^{1}H(s,\tau )a(\tau )f(u(\tau
))d\tau \Big)ds+\frac{\lambda +\gamma \mu h}{1-\gamma } \\
& =k^{\theta }Tu(t)+\frac{\lambda +\gamma \mu h}{1-\gamma } \\
& \geqslant k^{\theta }(Tu(t)+\frac{\lambda +\gamma \mu h}{1-\gamma }%
)=k^{\theta }T_{\lambda }u(t);
\end{align*}%
i.e., $T_{\lambda }$ is $\theta $-concave. By Lemma \ref{l4.1}, $T_{\lambda
} $ has a unique fixed point $u_{\lambda }$ in $P^{o}$, that is, the
fractional differential equation boundary-value problem (1) and (2) has a
unique positive solution. The proof is complete.
\end{proof}

\section{Nonexistence}

In this section, we let the Banach space $C[0,1]$ be endowed with the norm $%
\|u\|=\max_{0\leqslant t\leqslant1}|u(t)|$.

\begin{lemma}\label{l5.1} 
Assume {\rm (H1)} holds and let $0<\delta<1$ be given in {\rm (H4)}. 
Then the unique solution $u(t)$ of
fractional differential equation boundary-value problem (5)
and (6) satisfies
$$
u(t)\geqslant c_\delta\|u\|\quad \text{for }  \delta\leqslant t\leqslant1,
$$
where $c_\delta$ is defined by \eqref{e2.13}.
\end{lemma}

\begin{proof}
In view of Lemma 5 and Eq. (6), we have 
\begin{align*}
u(t)& \leqslant \int_{0}^{t}\frac{(t-s)^{\alpha -2}}{\Gamma (\alpha )}\phi
_{q}\Big(\int_{0}^{1}H(1,\tau )y(\tau )d\tau \Big)ds \\
& \quad +\frac{\gamma }{1-\gamma }\int_{0}^{h}\frac{(h-s)^{\alpha -2}}{%
\Gamma (\alpha )}\phi _{q}\Big(\int_{0}^{1}H(s,\tau )y(\tau )d\tau \big)ds+%
\frac{\lambda +\gamma \mu h}{1-\gamma } \\
& \leqslant \frac{1}{\Gamma (\alpha +1)}\phi _{q}\Big(\int_{0}^{1}H(1,\tau
)y(\tau )d\tau \Big) \\
& \quad +\frac{\gamma }{1-\gamma }\int_{0}^{h}\frac{(h-s)^{\alpha -2}}{%
\Gamma (\alpha )}\phi _{q}\Big(\int_{0}^{1}H(s,\tau )y(\tau )d\tau \Big)ds+%
\frac{\lambda +\gamma \mu h}{1-\gamma }
\end{align*}%
for $t\in \lbrack 0,1]$, and 
\begin{align*}
u(t)& \geqslant \int_{0}^{t}\frac{(t-s)^{\alpha -2}}{\Gamma (\alpha )}\phi
_{q}\Big(\int_{0}^{1}s^{\beta -1}H(1,\tau )y(\tau )d\tau \Big)ds \\
& \quad +\frac{\gamma }{1-\gamma }\int_{0}^{h}\frac{(h-s)^{\alpha -2}}{%
\Gamma (\alpha )}\phi _{q}\Big(\int_{0}^{1}H(s,\tau )y(\tau )d\tau \Big)ds+%
\frac{\lambda +\gamma \mu h}{1-\gamma } \\
& =\int_{0}^{t}\alpha (t-s)^{\alpha -2}\phi _{q}(s^{\beta -1})ds\frac{1}{%
\Gamma (\alpha +1)}\phi _{q}\Big(\int_{0}^{1}H(1,\tau )y(\tau )d\tau \Big) \\
& \quad +\frac{\gamma }{1-\gamma }\int_{0}^{h}\frac{(h-s)^{\alpha -2}}{%
\Gamma (\alpha )}\phi _{q}\Big(\int_{0}^{1}H(s,\tau )y(\tau )d\tau \Big)ds+%
\frac{\lambda +\gamma \mu h}{1-\gamma } \\
& \geqslant c_{\delta }\frac{1}{\Gamma (\alpha +1)}\phi _{q}\Big(%
\int_{0}^{1}H(1,\tau )y(\tau )d\tau \Big) \\
& \quad +\frac{\gamma }{1-\gamma }\int_{0}^{h}\frac{(h-s)^{\alpha -2}}{%
\Gamma (\alpha )}\phi _{q}\Big(\int_{0}^{1}H(s,\tau )y(\tau )d\tau \Big)ds+%
\frac{\lambda +\gamma \mu h}{1-\gamma } \\
& \geqslant c_{\delta }\big[\frac{1}{\Gamma (\alpha +1)}\phi _{q}\Big(%
\int_{0}^{1}H(1,\tau )y(\tau )d\tau \Big) \\
& \quad +\frac{\gamma }{1-\gamma }\int_{0}^{h}\frac{(h-s)^{\alpha -2}}{%
\Gamma (\alpha )}\phi _{q}\Big(\int_{0}^{1}H(s,\tau )y(\tau )d\tau \Big)ds+%
\frac{\lambda +\gamma \mu h}{1-\gamma }\big]
\end{align*}%
for $t\in \lbrack \delta ,1]$. Therefore, $u(t)\geqslant c_{\delta }\Vert
u\Vert $ for $\delta \leqslant t\leqslant 1$. The proof is complete.
\end{proof}

\begin{theorem}
Assume that (H1), (H4) hold. Then the fractional differential equation
boundary-value problem (1) and (2) has no positive solution for $\lambda
+\gamma \mu h>\left( 1-\gamma \right) e.$
\end{theorem}

\begin{proof}
Assume, to the contrary, the fractional differential equation boundary-value
problem (1) and (2) has a positive solution $u(t)$ for $\lambda +\gamma \mu
h>(1-\gamma )e$. Then by Lemma 4, we have 
\begin{align*}
u(t)& =\int_{0}^{t}\frac{(t-s)^{\alpha -2}}{\Gamma (\alpha )}\phi _{q}\Big(%
\int_{0}^{1}H(s,\tau )a(\tau )f(u(\tau ))d\tau \Big)ds \\
& \quad +\frac{\gamma }{1-\gamma }\int_{0}^{h}\frac{(h-s)^{\alpha -2}}{%
\Gamma (\alpha )}\phi _{q}\Big(\int_{0}^{1}H(s,\tau )a(\tau )f(u(\tau
))d\tau \Big)ds+\frac{\lambda +\gamma \mu h}{1-\gamma }
\end{align*}%
Therefore, $u(t)>e$ on [0,1]. In view of \eqref{e2.11} and \eqref{e2.12}, we
obtain 
\begin{gather*}
f(u(t))\geqslant N\phi _{p}(u(t))\quad \text{on }[0,1], \\
c_{\delta }\phi _{q}(N)\phi _{q}\Big(\int_{\delta }^{1}H(1,\tau )a(\tau
)d\tau \Big)\int_{0}^{1}\frac{(1-s)^{\alpha -2}}{\Gamma (\alpha )}\phi
_{q}(s^{\beta -1})ds>1.
\end{gather*}%
Then by Lemmas \ref{l2.5} and \ref{l5.1}, we obtain 
\begin{equation*}
\begin{array}{ll}
\Vert u\Vert =u(1) & >\int_{0}^{1}\frac{(1-s)^{\alpha -2}}{\Gamma (\alpha )}%
\phi _{q}\Big(\int_{0}^{1}H(s,\tau )a(\tau )f(u(\tau ))d\tau \Big)ds \\ 
& \geqslant \int_{0}^{1}\frac{(1-s)^{\alpha -2}}{\Gamma (\alpha )}\phi
_{q}(s^{\beta -1})ds\phi _{q}\Big(\int_{0}^{1}H(1,\tau )a(\tau )f(u(\tau
))d\tau \Big) \\ 
& \geqslant \int_{0}^{1}\frac{(1-s)^{\alpha -2}}{\Gamma (\alpha )}\phi
_{q}(s^{\beta -1})ds\phi _{q}(N)\phi _{q}\Big(\int_{\delta }^{1}H(1,\tau
)a(\tau )\phi _{p}(u(\tau ))d\tau \big) \\ 
& \geqslant \Vert u\Vert c_{\delta }\int_{0}^{1}\frac{(1-s)^{\alpha -2}}{%
\Gamma (\alpha )}\phi _{q}(s^{\beta -1})ds\phi _{q}(N)\phi _{q}\Big(%
\int_{\delta }^{1}H(1,\tau )a(\tau )d\tau \Big) \\ 
& >\Vert u\Vert .%
\end{array}%
\end{equation*}%
This contradiction completes the proof
\end{proof}

\begin{corollary}\label{c5.1} 
Assume that {\rm (H1)} holds and $f_\infty=+\infty$. Then
the fractional differential equation boundary-value problem (1)
and (2) has no positive solution for sufficiently large $\lambda>0$.
\end{corollary}

\section{Conclusion: Identities on the special polynomials whereby Caputo
Fractional derivative}

In this final part, we will focus on the new interesting identities \
related special polynomials by means of Caputo fractional derivative.

As well known, the Bernoulli polynomials may be defined to be:%
\begin{equation}
F\left( t,z\right) =\frac{z}{e^{z}-1}e^{tz}=e^{Bz}=\sum_{n=0}^{\infty
}B_{n}\left( t\right) \frac{z^{n}}{n!},  \tag{19}  \label{final eq. 1}
\end{equation}%
which usual convention about replacing $B^{n}$ by $B_{n}$ in, is used. Also,
we note that the Bernoulli polynomials is analytic on the region $D=\left\{
z\in 
\mathbb{C}
\mid \left\vert z\right\vert <2\pi \right\} $ (see \cite{Kim1}).

Let $\frac{d}{dt}$ be familiar normal derivative, says us the following
identity: 
\begin{equation}
\frac{d}{dt}t^{n}=nt^{n-1}.  \tag{20}  \label{final eq. 2}
\end{equation}

Differentiating in both sides of (\ref{final eq. 1}), we have%
\begin{equation}
\frac{d}{dt}B_{n}\left( t\right) =nB_{n-1}\left( t\right) \text{ }\left( 
\text{see \cite{Kim1}}\right) .  \tag{21}  \label{final eq. 3}
\end{equation}

When $t=0$ in (\ref{final eq. 1}), we have $B_{n}\left( 0\right) :=B_{n}$
are called Bernoulli numbers, which can be generated by%
\begin{equation}
F\left( z\right) =\frac{z}{e^{z}-1}=\sum_{n=0}^{\infty }B_{n}\frac{z^{n}}{n!}%
.  \tag{22}  \label{final eq. 4}
\end{equation}

By (\ref{final eq. 1}) and (\ref{final eq. 4}), we have the following
functional equation:%
\begin{equation*}
F\left( t,z\right) =e^{tz}F\left( z\right)
\end{equation*}%
and this equation yields to%
\begin{equation*}
B_{m}\left( t\right) =\sum_{k=0}^{m}\binom{m}{k}t^{m-k}B_{k}=\sum_{k=0}^{m}%
\binom{m}{k}t^{k}B_{m-k}\text{ (see \cite{Kim1})}
\end{equation*}

Let us now take $y\left( t\right) =B_{m}\left( t\right) $ in Definition 2,
leads to%
\begin{eqnarray*}
D_{0+}^{\alpha }B_{m}\left( t\right) &=&\frac{1}{\Gamma \left( n-\alpha
\right) }\int_{0}^{t}\frac{\frac{d^{n}}{dt^{n}}B_{m}\left( t\right) \mid
_{t=s}}{\left( t-s\right) ^{\alpha -n+1}}ds \\
&=&m\left( m-1\right) \cdots \left( m-n+1\right) \sum_{k=0}^{m-n}\binom{m-n}{%
k}B_{m-n-k}\left[ \frac{1}{\Gamma \left( n-\alpha \right) }\int_{0}^{t}\frac{%
s^{k}}{\left( t-s\right) ^{\alpha -n+1}}ds\right] \\
&=&\frac{\Gamma \left( m+1\right) }{\Gamma \left( m-n+1\right) }%
\sum_{k=0}^{m-n}\frac{k!\binom{m-n}{k}B_{m-n-k}}{\Gamma \left( n+k-\alpha
+1\right) }t^{k-\alpha +n}.
\end{eqnarray*}

Therefore, we procure the following theorem.

\begin{theorem}
The following identity holds true:%
\begin{equation*}
D_{0+}^{\alpha }B_{m}\left( t\right) =\frac{\Gamma \left( m+1\right) }{%
\Gamma \left( m-n+1\right) }\sum_{k=0}^{m-n}\frac{k!\binom{m-n}{k}B_{m-n-k}}{%
\Gamma \left( n+k-\alpha +1\right) }t^{k-\alpha +n}.
\end{equation*}
\end{theorem}

In \cite{Kim4}, The Bernoulli polynomials of higher order are defined by%
\begin{equation}
\underset{l-times}{\underbrace{\frac{z}{e^{z}-1}\frac{z}{e^{z}-1}\cdots 
\frac{z}{e^{z}-1}}}e^{tz}=\sum_{m=0}^{\infty }B_{m}^{\left( l\right) }\left(
t\right) \frac{z^{n}}{n!},  \tag{23}  \label{final eq. 5}
\end{equation}%
we note that $B_{m}^{\left( l\right) }\left( t\right) $ is analytic on $D.$
It follows from (\ref{final eq. 5}), we have 
\begin{equation}
\frac{d}{dt}B_{m}^{\left( l\right) }\left( t\right) =mB_{m-1}^{\left(
l\right) }\left( t\right) \text{ and }\frac{d^{n}}{dt^{n}}B_{m}^{\left(
l\right) }\left( t\right) =\frac{\Gamma \left( m+1\right) }{\Gamma \left(
m-n+1\right) }B_{m-n}^{\left( l\right) }\left( t\right) \text{ (see \cite%
{Kim4})}  \tag{24}  \label{final eq. 6}
\end{equation}

Substituting $x=0$ into (\ref{final eq. 5}), $B_{m}^{\left( l\right) }\left(
0\right) :=B_{m}^{\left( l\right) }$ are called Bernoulli polynomials of
higher order.

Owing to (\ref{final eq. 5}) and (\ref{final eq. 6}), we readily see that%
\begin{eqnarray*}
D_{0+}^{\alpha }B_{m}^{\left( l\right) }\left( t\right) &=&\frac{1}{\Gamma
\left( n-\alpha \right) }\int_{0}^{t}\frac{\frac{d^{n}}{dt^{n}}B_{m}^{\left(
l\right) }\left( t\right) \mid _{t=s}}{\left( t-s\right) ^{\alpha -n+1}}ds \\
&=&m\left( m-1\right) \cdots \left( m-n+1\right) \sum_{k=0}^{m-n}\binom{m-n}{%
k}B_{m-n-k}^{\left( l\right) }\left[ \frac{1}{\Gamma \left( n-\alpha \right) 
}\int_{0}^{t}\frac{s^{k}}{\left( t-s\right) ^{\alpha -n+1}}ds\right] \\
&=&\frac{\Gamma \left( m+1\right) }{\Gamma \left( m-n+1\right) }%
\sum_{k=0}^{m-n}\frac{k!\binom{m-n}{k}B_{m-n-k}^{\left( l\right) }}{\Gamma
\left( n+k-\alpha +1\right) }t^{k-\alpha +n} \\
&=&\frac{\Gamma \left( m+1\right) }{\Gamma \left( m-n+1\right) }%
\sum_{k=0}^{m-n}\frac{k!\binom{m-n}{k}}{\Gamma \left( n+k-\alpha +1\right) }%
t^{k-\alpha +n} \\
&&\times \left( \sum_{\underset{s_{l}\geq 0}{s_{1}+s_{2}+\cdots +s_{l}=m-n-k}%
}\binom{m-n-k}{s_{1},s_{2},\cdots ,s_{l}}\left(
\tprod\limits_{j=1}^{l}B_{s_{j}}\right) \right) .
\end{eqnarray*}

Therefore, we can state the following theorem.

\begin{theorem}
The following identity holds true:%
\begin{equation*}
D_{0+}^{\alpha }B_{m}^{\left( l\right) }\left( t\right) =\frac{\Gamma \left(
m+1\right) }{\Gamma \left( m-n+1\right) }\sum_{k=0}^{m-n}\frac{k!\binom{m-n}{%
k}}{\Gamma \left( n+k-\alpha +1\right) }t^{k-\alpha +n}\left( \sum_{\underset%
{s_{l}\geq 0}{s_{1}+s_{2}+\cdots +s_{l}=m-n-k}}\binom{m-n-k}{%
s_{1},s_{2},\cdots ,s_{l}}\left( \tprod\limits_{j=1}^{l}B_{s_{j}}\right)
\right) 
\end{equation*}%
in which $B_{s_{j}}$ and $\binom{m-n-k}{s_{1},s_{2},\cdots ,s_{l}}$are
Bernoulli numbers and multi-binomial coefficients.
\end{theorem}

In the region $T=\left\{ z\in 
\mathbb{C}
\mid \left\vert z\right\vert <\pi \right\} ,$ the Euler polynomials and the
Euler polynomials of higher order are given, respectively, with the help of
the following generating functions:%
\begin{eqnarray}
\frac{2}{e^{z}+1}e^{tz} &=&\sum_{m=0}^{\infty }E_{m}\left( t\right) \frac{%
z^{m}}{m!},  \TCItag{25}  \label{final eq. 7} \\
\underset{l-times}{\underbrace{\frac{2}{e^{z}+1}\frac{2}{e^{z}+1}\cdots 
\frac{2}{e^{z}+1}}}e^{tz} &=&\sum_{m=0}^{\infty }\left(
\sum_{s_{1}+s_{2}+\cdots +s_{l}=m}\binom{m}{s_{1},s_{2},\cdots ,s_{l}}\left(
\tprod\limits_{j=1}^{l-1}E_{s_{j}}\right) t^{s_{l}}\right) \frac{z^{m}}{m!} 
\notag \\
&=&\sum_{m=0}^{\infty }E_{m}^{\left( l\right) }\left( t\right) \frac{z^{n}}{%
n!},  \notag
\end{eqnarray}%
which $E_{s_{j}}$ are Euler numbers in (see \cite{Kim1}, \cite{Kim2}, \cite%
{Kim3} and \cite{Kim4}). From the last equation, we discover the followings:%
\begin{equation}
\frac{d}{dt}E_{m}\left( t\right) =mE_{m-1}\left( t\right) \text{ and }\frac{d%
}{dt}E_{m}^{\left( l\right) }\left( t\right) =mE_{m-1}^{\left( l\right)
}\left( t\right) \text{ (see \cite{Kim2}).}  \tag{26}  \label{final eq. 8}
\end{equation}

Obviously that%
\begin{equation*}
E_{m}^{\left( 1\right) }\left( t\right) :=E_{m}\left( t\right) .
\end{equation*}

Taking $y\left( t\right) =E_{m}^{\left( l\right) }\left( t\right) $ in
Definition 2, by (\ref{final eq. 7}) and (\ref{final eq. 8}), we compute%
\begin{eqnarray*}
D_{0+}^{\alpha }E_{m}^{\left( l\right) }\left( t\right) &=&\frac{1}{\Gamma
\left( n-\alpha \right) }\int_{0}^{t}\frac{\frac{d^{n}}{dt^{n}}E_{m}^{\left(
l\right) }\left( t\right) \mid _{t=s}}{\left( t-s\right) ^{\alpha -n+1}}ds \\
&=&m\left( m-1\right) \cdots \left( m-n+1\right) \sum_{k=0}^{m-n}\binom{m-n}{%
k}E_{m-n-k}^{\left( l\right) }\left[ \frac{1}{\Gamma \left( n-\alpha \right) 
}\int_{0}^{t}\frac{s^{k}}{\left( t-s\right) ^{\alpha -n+1}}ds\right] \\
&=&\frac{\Gamma \left( m+1\right) }{\Gamma \left( m-n+1\right) }%
\sum_{k=0}^{m-n}\frac{k!\binom{m-n}{k}E_{m-n-k}^{\left( l\right) }}{\Gamma
\left( n+k-\alpha +1\right) }t^{k-\alpha +n} \\
&=&\frac{\Gamma \left( m+1\right) }{\Gamma \left( m-n+1\right) }%
\sum_{k=0}^{m-n}\frac{k!\binom{m-n}{k}}{\Gamma \left( n+k-\alpha +1\right) }%
t^{k-\alpha +n} \\
&&\times \left( \sum_{\underset{s_{l}\geq 0}{s_{1}+s_{2}+\cdots +s_{l}=m-n-k}%
}\binom{m-n-k}{s_{1},s_{2},\cdots ,s_{l}}\left(
\tprod\limits_{j=1}^{l}E_{s_{j}}\right) \right) .
\end{eqnarray*}

Therefore, we obtain the following theorem.

\begin{theorem}
The following identity%
\begin{equation*}
D_{0+}^{\alpha }E_{m}^{\left( l\right) }\left( t\right) =\frac{\Gamma \left(
m+1\right) }{\Gamma \left( m-n+1\right) }\sum_{k=0}^{m-n}\frac{k!\binom{m-n}{%
k}}{\Gamma \left( n+k-\alpha +1\right) }\left( \sum_{\underset{s_{l}\geq 0}{%
s_{1}+s_{2}+\cdots +s_{l}=m-n-k}}\binom{m-n-k}{s_{1},s_{2},\cdots ,s_{l}}%
\left( \tprod\limits_{j=1}^{l}E_{s_{j}}\right) \right) t^{k-\alpha +n}.
\end{equation*}%
is true. Obviously that%
\begin{equation*}
D_{0+}^{\alpha }E_{m}\left( t\right) =\frac{\Gamma \left( m+1\right) }{%
\Gamma \left( m-n+1\right) }\sum_{k=0}^{m-n}\frac{k!\binom{m-n}{k}E_{m-n-k}}{%
\Gamma \left( n+k-\alpha +1\right) }t^{k-\alpha +n}.
\end{equation*}
\end{theorem}

In the region $T=\left\{ z\in 
\mathbb{C}
\mid \left\vert z\right\vert <\pi \right\} ,$ Genocchi polynomials, $%
G_{m}\left( x\right) $, and Genocchi polynomials of higher order, $%
G_{m}^{\left( l\right) }\left( x\right) $, are defined as an extension of
Genocchi numbers $G_{m}$ defined in \cite{Kim1}, \cite{Araci1}, \cite{seo}.,
respectively:%
\begin{eqnarray}
\frac{2z}{e^{z}+1}e^{tz} &=&\sum_{m=0}^{\infty }G_{m}\left( t\right) \frac{%
z^{m}}{m!},  \TCItag{27}  \label{final eq. 9} \\
\underset{l-times}{\underbrace{\frac{2z}{e^{z}+1}\frac{2z}{e^{z}+1}\cdots 
\frac{2z}{e^{z}+1}}}e^{tz} &=&\sum_{m=0}^{\infty }G_{m}^{\left( l\right)
}\left( t\right) \frac{z^{n}}{n!}.  \notag
\end{eqnarray}

By the similar method, in this final section, we arrive at the following
theorem.

\begin{theorem}
The following identity%
\begin{equation*}
D_{0+}^{\alpha }G_{m}^{\left( l\right) }\left( t\right) =\frac{\Gamma \left(
m+1\right) }{\Gamma \left( m-n+1\right) }\sum_{k=0}^{m-n}\frac{k!\binom{m-n}{%
k}}{\Gamma \left( n+k-\alpha +1\right) }\left( \sum_{\underset{s_{l}\geq 0}{%
s_{1}+s_{2}+\cdots +s_{l}=m-n-k}}\binom{m-n-k}{s_{1},s_{2},\cdots ,s_{l}}%
\left( \tprod\limits_{j=1}^{l}G_{s_{j}}\right) \right) t^{k-\alpha +n},
\end{equation*}%
is true. Obviously that,%
\begin{equation*}
D_{0+}^{\alpha }G_{m}\left( t\right) =\frac{\Gamma \left( m+1\right) }{%
\Gamma \left( m-n+1\right) }\sum_{k=0}^{m-n}\frac{k!\binom{m-n}{k}G_{m-n-k}}{%
\Gamma \left( n+k-\alpha +1\right) }t^{k-\alpha +n}.
\end{equation*}
\end{theorem}

\end{document}